\documentclass{daj}

\dajAUTHORdetails{%
  title = {New bounds on curve tangencies and orthogonalities}, 
  author = {Jordan S. Ellenberg, Jozsef Solymosi, and Joshua Zahl},
  plaintextauthor = {Jordan S. Ellenberg, Jozsef Solymosi, Joshua Zahl},
  plaintexttitle = {New bounds on curve tangencies and orthogonalities}, 
  copyrightauthor = {J. Ellenberg, J. Solymosi, J. Zahl},
}   

\dajEDITORdetails{%
   year={2016},
   number={18},
   received={1 April 2016},   
   revised={28 October 2016},    
   published={4 November 2016},  
   doi={10.19086/da990},       
}   

\usepackage{color,latexsym,amsmath,amssymb,path, amsthm, graphicx, amssymb}
\newtheorem{thm}{Theorem}
\newtheorem{cor}{Corollary}
\newtheorem{lem}{Lemma}

\theoremstyle{remark}
\newtheorem{defn}{Definition}
\newtheorem{rem}{Remark}
\newtheorem{example}{Example}

\newtheorem{conj}{Conjecture}

\newcommand{\RR}{\mathbb R}

\newcommand{\CC}{\mathbb C}

\newcommand{\BZ}{\mathbf{Z}}
\newcommand{\BI}{\mathbf{I}}

\renewcommand{\dim}{\operatorname{dim}}

\newcommand{\PC}{\mathcal{L}}

\newcommand{\beq}{\begin{displaymath}}
\newcommand{\eeq}{\end{displaymath}}
\newcommand{\del}{\partial}
\DeclareMathOperator{\car}{char}
\newcommand{\OO}{\mathcal{O}}

\begin{document}

\begin{frontmatter}[classification=text]

\title{New bounds on curve tangencies and orthogonalities} 

\author[JSE]{Jordan S. Ellenberg\thanks{Supported by NSF-DMS award 1402620, the John S. Guggenheim Fellowship, and the Vilas Distinguished Achievement Professorship.}}
\author[JS]{Jozsef Solymosi\thanks{Supported by NSERC, ERC Advanced Research Grant AdG. 321104, and by Hungarian National Research Grant NK 104183.}}
\author[JZ]{Joshua Zahl\thanks{Supported by a NSF Postdoctoral Fellowship.}}

\begin{abstract}
We establish new bounds on the number of tangencies and orthogonal intersections determined by an arrangement of curves. First, given a set of $n$ algebraic plane curves, we show that there are $O(n^{3/2})$ points where two or more curves are tangent. In particular, if no three curves are mutually tangent at a common point, then there are $O(n^{3/2})$ curve-curve tangencies. Second, given a family of algebraic plane curves and a set of $n$ curves from this family, we show that either there are $O(n^{3/2})$  points where two or more curves are orthogonal, or the family of curves has certain special properties. 

We obtain these bounds by transforming the arrangement of plane curves into an arrangement of space curves so that tangency (or orthogonality) of the original plane curves corresponds to intersection of space curves. We then bound the number of intersections of the corresponding space curves. For the case of curve-curve tangency, we use a polynomial method technique that is reminiscent of Guth and Katz's proof of the joints theorem. For the case of orthogonal curve intersections, we employ a bound of Guth and the third author to control the number of two-rich points in space curve arrangements.
\end{abstract}
\end{frontmatter}

\section{Introduction}
We will bound the number of tangencies and orthogonal intersections determined by a set of $n$ algebraic plane curves. Bounding the maximal number of curve tangencies plays an important role in combinatorial and computational geometry.  For the relevant works and an extended bibliography of the subject, we refer to the classical paper of Agarwal et.~al.~\cite{Agarwal} and the recent publication of Pach et.~al.~\cite{PachRT15}. The latter also deals in part with tangencies between bounded degree algebraic curves, which is the main subject of our paper.

\subsection{Tangent curves}
If $n$ curves are mutually tangent at a common point, then this would lead to $\binom{n}{2}\sim n^2$ tangencies. To avoid this degenerate situation, we could require that no three curves be mutually tangent at a common point. Instead we will count a slightly different quantity. 

\begin{defn}[Directed points of tangency]\label{directedPointOfTangency}
Let $k$ be a field and let $\PC$ be a set of irreducible algebraic curves in $k^2$. Let $p$ be a point in $k^2$ and let $\ell$ be a line passing through $p$. We say that $(p,\ell)$ is a directed point of tangency for $\PC$ if there are at least two distinct curves in $\PC$ that are smooth at $p$ and tangent to $\ell$ at $p$ (for readers unfamiliar with algebraic geometry, the definition of a smooth point of a curve is given in Section \ref{planeCurveSec}). Define $\mathcal{T}(\mathcal{L})$ to be the set of directed points of tangency, and for each $(p,\ell)\in\mathcal{T}(\mathcal{L}),$ define $\operatorname{mult}(p,\ell;\mathcal{L})$ to be the number of curves from $\mathcal{L}$ that are smooth at $p$ and tangent to $\ell$ at $p$.
\end{defn}

Our first main result is the following bound on curve tangencies.

\begin{thm}\label{mainThm}
Let $D\geq 1$ be an integer. Then there are constants $c_D>0$ (small) and $C_D$ (large) so that the following holds. Let $k$ be a field, and let $\mathcal{L}$ be a set of $n$ irreducible plane curves in $k^2$ of degree at most $D$. Suppose that $n\leq c_D \operatorname{char}(k)^2$ (if $\operatorname{char}(k)=0$ then we place no restrictions on $n$). Then 
\begin{equation}
\sum_{(p,\ell)\in\mathcal{T}(\mathcal{L})}\operatorname{mult}(p,\ell;\mathcal{L})\leq C_Dn^{3/2}.
\end{equation}

\end{thm}

\begin{cor}
A set of $n$ (real or complex) plane algebraic curves of degree at most $D$, no three of which pass through a common point, determine $O_D(n^{3/2})$ curve-curve tangencies.
\end{cor}

We will prove Theorem \ref{mainThm} using the algebraic techniques first pioneered by Dvir in \cite{dvir09} to solve the finite field Kakeya problem, and later used by Guth and Katz in \cite{GuKa} to solve the joints problem. We will describe a process that converts plane curves to space curves, so that curve-curve tangency in the plane $k^2$ corresponds to curve-curve intersection in $k^3$. We will then use a joints-like argument to show that not too many curve-curve intersections can occur.

\subsection{Orthogonal curves}
We will also obtain a bound on the number of pairs of curves that can intersect orthogonally. However, a problem immediately arises: it is possible for $n$ lines to determine $n^2/4$ points where two lines intersect orthogonally---just select a set of $n/2$ parallel lines and a second set of $n/2$ parallel lines orthogonal to the first set. Similarly, it is possible to find two sets of circles, each of cardinality $n/2$, so that each circle from the first set intersects each circle from the second set orthogonally; such arrangements are called Apollonian circles. These are two simple examples of what are known as orthogonal trajectories.

In the examples above, $n$ curves can determine $\sim n^2$ orthogonal intersections. We will establish a certain dichotomy: given a particular family of curves (defined below), either every arrangement of curves from this family determines very few orthogonal intersections, or it is possible to find arrangements with a nearly maximal number of orthogonal intersections. This will be stated precisely in Theorem \ref{orthThm} below. First however, we will need to introduce several definitions.

The projective space $\mathbf{P}(k^d)$ is the quotient of $k^d\backslash\{0\}$ by the equivalence relation that identifies $(w_1,\ldots,w_d)$ with $(w_1^\prime,\ldots,w_d^\prime)$ if there is some non-zero $\lambda\in k$ such that $w_i = \lambda w_i^\prime$ for each index $i$. We will often omit the brackets and write $\mathbf{P}k^d$ in place of $\mathbf{P}(k^d)$. A set $X\subset \mathbf{P}k^d$ is an (not necessarily irreducible) algebraic variety if it can be written as the zero-set of a collection of homogeneous polynomials in $d$ variables. If $D\leq\operatorname{char}(k)$, then the set of algebraic curves in $k^2$ of degree $\leq D$ are in one-to-one correspondence with the set of points in $\mathbf{P}k^{\binom{D+2}{2}}$. We will frequently abuse notation and identify these two sets. In particular, if $X\subset \mathbf{P}k^{\binom{D+2}{2}}$ is a variety, then we will abuse notation and refer to $X$ as a family of degree $\leq D$ curves. 

\begin{defn}[Directed points of orthogonality]\label{directedPointOfTangency}
Let $k$ be a field and let $\PC$ be a set of irreducible algebraic curves in $k^2$. Let $p$ be a point in $k^2$, let $\ell$ be a line passing through $p$, and let $\ell^\perp$ be the line passing through $p$ that is orthogonal to $\ell$ (here the vector $(v_1,v_2)$ is orthogonal to the vector $(v_1^\prime,v_2^\prime)$ if $v_1v_1^\prime+v_2v_2^\prime=0$). We say that $(p,\ell)$ is a directed point of orthogonality for $\PC$ if there is a curve in $\PC$ that is smooth at $p$ and tangent to $\ell$ at $p$, and there is a second curve in $\PC$ that is smooth at $p$ and tangent to $\ell^\perp$ at $p$.
\end{defn}

\begin{thm}\label{orthThm}
Let $k$ be a field and let $X\subset \mathbf{P}k^{\binom{D+2}{2}}$ be a family of degree $\leq D$ curves in $k^2$. Then exactly one of the following must hold.
\begin{itemize}
 \item Every set of $n\leq c_{D} (\operatorname{char} k)^2$ curves from $X$ determines $O_{D,X}(n^{3/2})$ directed points of orthogonality. 
 \item Let $K$ be the algebraic closure of $k$, and let $\bar X\subset \mathbf{P}K^{\binom{D+2}{2}}$ be the closure of $X$. Then for each $1\leq n\leq c_{D,X}(\operatorname{char} k)^2,$ we can find $n$ curves in $\bar X$ that determine $\frac{n^2}{4}(1-o_{D,X}(1))$ directed points of orthogonality\footnote{The expression $o_{D,X}(1)$ refers to a term that tends to 0 as $n$ (and thus $\operatorname{char} k$) tends to infinity, while the degree $D$ and the equations defining the variety $X$ are kept fixed.}.
\end{itemize}
\end{thm}

To prove Theorem \ref{orthThm}, we will again convert plane curves to space curves so that curve-curve orthogonality in the plane corresponds to curve-curve intersection in $k^3$. We will then apply Theorem 1.2 from \cite{GuZa}. This theorem states that for a given family of space curves, either (i) any set of curves from this family determines few curve-curve intersections in $k^3$, or (ii) there are many curves from the family that form two pairwise intersecting sets of curves. Case (i) corresponds to few directed points of orthogonality, while case (ii) allows us to construct arrangements of $n$ curves with $\frac{n^2}{4}(1-o_{D,X}(1))$ directed points of orthogonality.

\subsection{Previous work}
In \cite{Clarkson}, Clarkson et.~al.~developed techniques to bound the number of incidences between points and surfaces in $\RR^3$. In \cite{Wolff97, Wolff00}, Wolff observed\footnote{The tangency bound for circles is not stated explicitly in \cite{Wolff97, Wolff00}, but they are discussed in \cite[Section 3]{Wolff96}, which is an expository paper that discusses the results in \cite{Wolff97, Wolff00}} that these techniques could be used to bound the number of tangencies between circles in $\RR^2$.  Specifically, $n$ circles in $\RR^2$ determine $O(n^{3/2}\beta(n))$ directed points of tangency, where $\beta(n)$ is a very slowly growing function. Some of the methods from \cite{zahl13} can also be used to slightly improve this tangency bound to $O(n^{3/2})$. All of these methods, however, made crucial use of two facts. First, circles in the plane have three degrees of freedom---a circle can be described by three parameters. While these techniques can be extended to general degree $D$ algebraic curves in $\RR^2$, the resulting exponent in the bound becomes worse as $D$ increases. Second, these older techniques make crucial use of various ``cuttings'' or ``polynomial partitioning'' theorems, and these results generally cannot be extended to fields other than $\RR$. Theorem \ref{mainThm} suffers from neither of these constraints.

In \cite{MeSz}, Megyesi and Szab\'o proved that if $k$ is a field whose characteristic is not two, then $n$ conics in $k^2$ determine $O(n^{2-1/7633})$ directed points of tangency. Theorem \ref{mainThm} improves this to $O(n^{3/2})$ in the special case where $\operatorname{char} k=0$ or $n < C(\operatorname{char} k)^2$ for a suitable (absolute) constant $C$.  

In \cite{MaTa}, Marcus and Tardos showed that any family of $n$ pseudocircles determines $O(n^{3/2}\log n)$ tangencies. In the case where the pseudocircles are defined by bounded-degree algebraic curves, Theorem \ref{mainThm} improves this bound. This is also Problem 14 in Chapter 7.1 of Brass-Moser-Pach \cite{BMP}.

Purely combinatorial methods work well for bounding pseudocircle tangencies. However, one can not expect general results for ``pseudocubics'' as the following simple construction shows. In Figure \ref{fig:cubics}  we present an arrangement of curves in $\RR^2$ where every two curves meet in two points and every two curves are tangent; this results in $\binom{n}{2}$ directed points of tangency. Thus the requirement that the curves be algebraic (and of controlled degree) is crucial. \\
$\phantom{1}$

\begin{figure} [h]
    \centering
\includegraphics[trim=0.0cm 3cm 0.0cm 4cm, width=10cm]{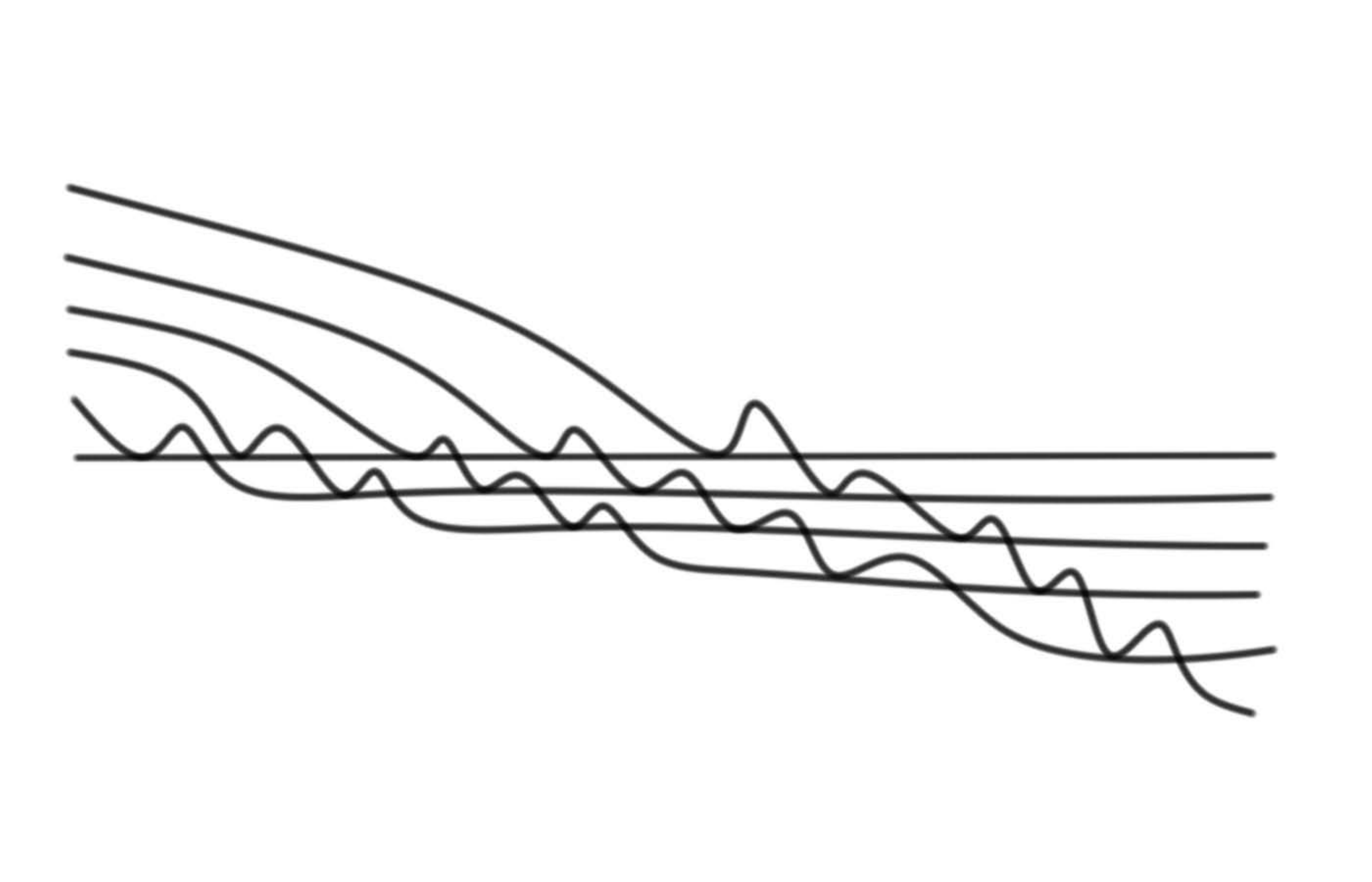}
\caption{Pseudocubics}
    \label{fig:cubics}
\end{figure}

\subsection{Proof sketch}
In this section we will sketch the proof of Theorem \ref{mainThm} in the special case where $k=\RR$, the curves in $\mathcal{L}$ are circles, and no three circles are tangent at a common point. Suppose that there are more than $Cn^{3/2}$ tangencies. If $C$ is sufficiently large, we will obtain a contradiction. To simply the proof sketch slightly, we will assume that each circle is tangent to at least $(C/2)n^{1/2}$ other circles (we can always reduce to this case by refining our collection of circles slightly). After applying a rotation, we can assume that for every directed point of tangency $(p,\ell)$, the line $\ell$ does not point in the $y$--direction.

 For each circle $\gamma\in\mathcal{L}$, Define
$$
\beta[\gamma]=\Big\{(x,y,z)\in\RR^3\colon (x,y)\in\gamma,\ z = -\frac{x-x_0}{y-y_0}\Big\},
$$
where $(x_0,y_0)$ is the center of $\gamma$. The reason for this definition is as follows. Let $(x,y)\in\RR^2$ and let $\ell$ be a non-vertical line containing $(x,y)$. Let $z$ be the slope of $\ell$. Then $\gamma$ is tangent to $\ell$ at the point $(x,y)$ if and only if $(x,y,z)\in\beta[\gamma]$. In particular, the projection of $\beta[\gamma]$ to the $xy$--plane is the set $\gamma\backslash\{p_1,p_2\}$, where $p_1$ and $p_2$ are the two points where $\gamma$ is tangent to a vertical line.

Define $\mathcal{C}=\{\beta[\gamma]\colon \gamma\in\mathcal{L}\}.$ Observe that $\gamma_1$ and $\gamma_2$ are tangent if and only if $\beta[\gamma_1]\cap \beta[\gamma_2]\neq\emptyset$ (recall that no two curves from $\mathcal{L}$ share a directed point of tangency $(p,\ell)$ where $\ell$ is a vertical line). Next, suppose that $(x,y,z)\in \beta[\gamma_1]\cap \beta[\gamma_2]$. If $\gamma_1\neq\gamma_2$, then 
\begin{equation}\label{001InSpan}
(0,0,1)\in \operatorname{span}\big(T_{(x,y,z)} \beta[\gamma_1],\ T_{(x,y,z)} \beta[\gamma_2]\big).
\end{equation}
To see this, define $f_i(t)$  so that for all $t$ in a small neighborhood of 0, $(t+x, f_i(t))$ is a parameterization of $\gamma_i$ in a neighborhood of $(x,y)$. Since $\gamma_1$ and $\gamma_2$ are tangent at $(x,y)$, $\frac{df_1}{dt}(0)=\frac{df_2}{dt}(0)$. Since the circles $\gamma_1$ and $\gamma_2$ are distinct, $\frac{d^2f_1}{dt^2}(0)\neq \frac{d^2f_2}{dt^2}(0)$. Next, observe that in a neighborhood of $(x,y,z)$, the curve $\beta[\gamma_i]$ is parameterized by $(t, f_i(t), \frac{df_i}{dt}(t))$. In particular, the vector $\big(1, \frac{df_i}{dt}(0), \frac{d^2f_i}{dt^2}(0)\big)$ is contained in the vector space $T_{(x,y,z)} \beta[\gamma_i]$. Thus $(0,0,1)$ is in the span of the vector 
$$
\Big(1, \frac{df_1}{dt}(0), \frac{d^2f_1}{dt^2}(0)\Big)-\Big(1, \frac{df_2}{dt}(0), \frac{d^2f_2}{dt^2}(0)\Big),
$$
which is contained in $\operatorname{span}\big(T_{(x,y,z)} \beta[\gamma_1],\ T_{(x,y,z)} \beta[\gamma_2]\big).$ This establishes \eqref{001InSpan}.

Next, let $P[x,y,z]$ be the non-zero polynomial of minimal degree that vanishes on all of the curves in $\mathcal{C}$. In particular, $\deg(P)=O(n^{1/2})$. By \eqref{001InSpan}, we have that if $(x,y,z)$ is a point in $\RR^3$ where two curves from $\mathcal{C}$ intersect, then $\partial_z P(x,y,z)=0$. Thus since each curve $\gamma\in\mathcal{L}$ is tangent to at least $(C/2)n^{1/2}$ other curves from $\mathcal{L}$, and each of these tangencies occurs at a distinct point of $\gamma$, we have that $\partial_zP$ vanishes at $\geq (C/2)n^{1/2}$ points on each curve $\beta\in\mathcal{C}$. If $C=O(1)$ is sufficiently large, then by B\'ezout's theorem we have that $\partial_zP$ vanishes on each curve from $\mathcal{C}$. But since $P$ was the non-zero polynomial of minimal degree that vanishes on all of the curves in $\mathcal{C}$, we conclude that $\partial_zP=0$, i.e. $P(x,y,z)=Q(x,y)$ for some bivariate polynomial $Q(x,y)$ of degree $O(n^{1/2})$. However, this immediately implies that each of the $n$ circles in $\mathcal{L}$ must be contained in $Z(Q)$. This is impossible, since $Q$ has degree $O(n^{1/2})$, while the algebraic curve $\bigcup_{\gamma\in\Gamma}\gamma$ has degree $2n$. We conclude that $\mathcal{L}$ determines fewer than $Cn^{3/2}$ tangencies.

When proving Theorem \ref{mainThm}, we must overcome two difficulties that are absent in the above proof sketch. First, since we are working over an arbitrary field, we must make sense of what it means to parameterize $\gamma$ in a  neighborhood of $t$. Second, it need not be the case that $\frac{d^2f_1}{dt^2}(0)\neq \frac{d^2f_2}{dt^2}(0)$; if the curves are not circles, then it is possible for two curves to be tangent to order higher than two. However, since the curves are algebraic of degree $\leq D$, the derivatives of $f_1$ and $f_2$ are 0 must eventually differ. Sections \ref{algGeoSec}--\ref{planeSpaceCurveSec} will develop the tools and techniques needed to make these ideas precise, and then in Section \ref{proofOfMainThmSec} we will prove Theorem \ref{mainThm}. Finally, in Section \ref{orthCurvesSec} we will prove Theorem \ref{orthThm}.

\section{Algebraic geometry preliminaries}\label{algGeoSec}

We will use several terms and tools from commutative algebra and algebraic geometry. The results discussed in this section can be found in standard texts on commutative algebra, such as \cite{Eisenbud}. In this section, all fields are algebraically closed (in Section \ref{planeSpaceCurvesSec}, we will explain why this assumption is harmless), and all rings $R$ are commutative and contain a multiplicative unit. 
\subsection{Power series and B\'ezout's theorem}
Define $k[x_1,x_2,\ldots,x_d]$ to be the ring of polynomials in $d$ variables that have coefficients in $k$ and define $k[[x_1,x_2,\ldots,x_d]]$ to be the ring of formal power series in $d$ variables that have coefficients in $k$. Let $R$ be a commutative ring and let $I\subset R$ be an ideal. We define the quotient $R/I$ to be the equivalence class of objects $\{\bar a\colon a\in R\}$, where $\bar a = \bar b$ if $a=b+r$ for some $r\in I$.

Note that there is a natural injection $k[x_1,\ldots,x_d]\to k[[x_1,\ldots,x_d]]$. Let $I\subset k[x_1,\ldots,x_d]$ be an ideal. Abusing notation, we will identify $I$ with the ideal generated by its image in $k[[x_1,\ldots,x_d]].$  Now $k[[x_1,\ldots,x_d]]/I$ is a $k$-vector space;  we denote its dimension, which may be finite or $\infty$, by $\operatorname{dim}_k k[[x_1,\ldots,x_d]]/I$.
 
\begin{thm}[B\'ezout]\label{BezoutThm}
Let $P,P^\prime\in k[x,y]$ be polynomials and let $(P,P^\prime)\subset k[x,y]$ be the ideal generated by $P$ and $P^\prime$. If $\operatorname{dim}_k  k[[x,y]]/(P,P^\prime)>(\deg P)(\deg P^\prime),$ then $P$ and $P^\prime$ share a common factor.
\end{thm}

Geometrically, one should think of the quantity $\operatorname{dim}_k  k[[x,y]]/(P,P^\prime)$ as the multiplicity of the intersection between the vanishing loci $\BZ(P)$ and $\BZ(P')$ at the point $(0,0)$; for instance, a ``typical" intersection between two curves will have multiplicity one, an intersection point where two curves are tangent has multiplicity $2$, and an intersection involving a higher degree of tangency will have still higher multiplicity.  B\'ezout's theorem says that the total multiplicity of the intersection, summed over {\em all} intersection points, between $\BZ(P)$ and $\BZ(P')$ is $(\deg P)(\deg P^\prime)$, as long as $P$ and $P'$ have no common factor; this provides the claimed upper bound for intersection multiplicity at the point $(0,0)$.

We will also use a second, less technical variant of B\'ezout's theorem, which states that if two varieties of degrees $D$ and $D^\prime$ intersect properly (i.e.~if the codimension of the intersection is the sum of the codimensions of the varieties), then intersection is a variety of degree at most $DD^\prime$.

\subsection{Derivatives and tangent spaces}
Let $P\in k[x_1,\ldots,x_d]$ be a polynomial. We denote the (formal) derivative of $P$ in the $x_j$ variable by $\partial_{x_j}P$. If $P/Q$ is a rational function, we define the (formal) derivative $\partial_{x_j}(P/Q)$ using the Leibniz rule $\partial_{x_j}(P/Q)=((\partial_{x_j}P)Q-P\partial_{x_j}Q)/Q^2$. If $P\in k[[x_1,\ldots,x_d]]$, then $\partial_{x_j}P$ will denote the formal derivative of $P$. Note that we will never use Newton's symbol $\prime$ to denote derivatives.  By $\nabla P$ we mean the vector $(\partial_{x_1} P, \ldots, \partial_{x_d} P)$.

Let $Z\subset k^d$ be a variety and let $p\in Z$. Let $\BI(Z)$ be the ideal of polynomials that vanish on $Z$. We define the the Zariski cotangent space of $Z$ at $p$ to be the span of the vectors $\{\nabla f(p)\colon f\in \BI(Z)\}$. If $\BI(Z)=(f_1,\ldots,f_\ell)$, then this will be equal to the span of $\nabla f_1(p),\ldots\nabla f_\ell(p)$. The Zariski tangent space $T_p(Z)$ is the dual of the Zariski cotangent space, i.e.
\begin{equation*}
T_p(Z)=\Big\{v\in k^d\colon \sum_{i=1}^d\partial_{d_{x_i}} f (p) v_i=0\ \textrm{for all}\ f\in I(Z)\Big\}.
\end{equation*}

In particular, note that the Zariski tangent space is a vector space, i.e.~if $v_1$ and $v_2$ are vectors in $T_p(Z),$ then so is $a_1v_1+a_2v_2$ for any $a_1,a_2\in k$.

We can also compute the Zariski tangent space by working in the power series ring.  Suppose $p=0\in k^d,$ and let $I \subset k[[x_1,\ldots,x_d]]$ be the image of $\BI(Z)$ in $k[[x_1,\ldots,x_d]]$. Then the Zariski cotangent space  $Z$ at $p=(0,0)$ is the span of the vectors $\{\nabla f(p)\colon f\in I\}$, and the Zariski tangent space of $Z$ at $p=(0,0)$ is given by
\begin{equation*}
\Big\{v\in k^d\colon \sum_{i=1}^d\partial_{d_{x_i}} f (p) v_i=0\ \textrm{for all}\ f\in I\Big\}.
\end{equation*}

If $p\neq 0$, then we can apply a translation sending $p$ to the origin and the above statement holds.

\subsection{Plane curves}\label{planeCurveSec}
Let $P\in k[x,y]$ be a polynomial, and let $\zeta=\BZ(P)=\{(x,y)\in k^2\colon P(x,y)=0\}$. Sets of this form will be called algebraic curves.  If $P$ is irreducible, we say the curve $\zeta$ is irreducible as well. Let $\zeta=\BZ(P)$ be an irreducible curve. We say a point $(x,y)\in\zeta$ is singular if $\partial_x P(x,y)=\partial_y P(x,y)=0$ (or, in other words, if the Zariski tangent space to $\zeta$ is two-dimensional.)  If $(x,y)\in \zeta$ is not a singular point, then we say it is a smooth point.  In this case, the tangent space to $\zeta$ at $(x,y)$ is one-dimensional, and specifies the unique tangent direction there.

Let $\zeta$ be an irreducible curve. We say a set $X\subset \zeta$ is \emph{Zariski dense} in $\zeta$ if $\zeta\backslash X$ is finite. We will sometimes make statements like ``a generic point on the irreducible curve $\zeta$ has the following properties....'' What this means is that all but a finite set of points $\mathcal{Q}\subset\zeta$ have the properties in question; the set $\mathcal{Q}$ depends only on $\zeta$ and the properties under consideration. For example, a generic point on $\zeta\subset k^2$ is smooth. This is because there is a finite set $\mathcal{Q}\subset\zeta$ of singular points, and every point in $\zeta\backslash\mathcal{Q}$ is smooth.

\subsection{Coordinates}
We will mainly work in $k^2$ and $k^3$. We will use coordinates $(x,y)$ to represent points in $k^2$, and $(x,y,z)$ to represent points in $k^3$. Unless otherwise noted, $\pi$ will be the projection $\pi(x,y,z)=(x,y)$. Sometimes we will restrict the domain of $\pi$ to a space curve in $k^3$, so we will refer to a projection $\pi\colon\gamma\to\zeta$, and we will say that ``a generic fiber of this projection has the following properties....'' This means that the fiber of $\pi$ above a generic point of $\zeta$ has the properties in question. In particular, the specified property holds for all but finitely many fibers.

For simplicity, we will often work in coordinates. Thus it may appear that the coordinate axes play a distinguished role. However, all of the quantities we study are invariant under invertible affine transformations. Thus we will sometimes apply a suitable transformation to our curve arrangement to ensure that no coincidences involving the coordinate axes occur. In particular, we will refer to a ``generic linear transformation'' in $k^d$ (generally $d=2$ or $3$). What this means is that there is a Zariski-open subset $X\subset \operatorname{GL}_d(k)$ that depends only on the curves from the statement of Theorem \ref{mainThm} or \ref{orthThm}, and we can select any linear transformation from this set.

\section{Plane and space curves}\label{planeSpaceCurvesSec}
In this section we will establish notation and prove some results that are useful for both Theorems \ref{mainThm} and \ref{orthThm}. Later, the two proofs will diverge, and we will handle the two theorems in separate sections. First, note that in Theorems \ref{mainThm} and \ref{orthThm}, we can assume without loss of generality that $k$ is algebraically closed. Indeed, if $k$ is not closed then we can replace $k$ by its algebraic closure.

\subsection{Notation for plane and space curves}\label{planeCurveDefPolySec}
Let $k[x,y]_{\leq D}$ be the vector space of bivariate polynomials of degree at most $D$; we can identify this vector space with $k^{\binom{D+2}{2}}$. For each $P\in k[x,y]_{\leq D}$, let $\alpha_P\in k^{\binom{D+2}{2}}$ be the corresponding point, and for each $\alpha\in k^{\binom{D+2}{2}}$, let $P_\alpha$ be the corresponding polynomial. We say that $P\in k[x,y]_{\leq D}$ is $x$--\emph{monic} if the coefficient of $x^\ell$ is 1, where $\ell=\deg(P)$. Given a finite set of irreducible plane curves of degree $\leq D<\operatorname{char}(k)$, we can always find an invertible affine transformation $k^2\to k^2$ so that after applying the transformation to the curves, each curve can be uniquely written as $\BZ(P)$, where $P\in k[x,y]_{\leq D}$ is $x$--monic and irreducible. If $\zeta\subset k^2$ is an irreducible curve of degree $\leq D$ that can be written as the zero-set of an irreducible $x$--monic polynomial, then define $P[\zeta]$ to be the corresponding irreducible $x$--monic polynomial, and define $\alpha_\zeta = \alpha_{P[\zeta]}$. 

At several points we will construct space curves that capture relevant properties from the plane curve arrangement. An important notion for an arrangement of space curves is that of a \emph{two-rich point}, which we define below.
\begin{defn}[Two-rich point]\label{defnTwoRich}
Let $\Gamma$ be a set of irreducible algebraic space curves. We say a point in $k^3$ is \emph{two-rich} (with respect to $\Gamma$) if there are two distinct curves from $\Gamma$ that contain the point.
\end{defn}

\subsection{Implicit differentiation from the formal viewpoint}\label{implicitDifferentiationSec}

Let $k$ be a field and let $P$ be an irreducible polynomial in $k[x,y]$ of degree $D$.  Let $\zeta=\BZ(P)$ be the vanishing locus of $P$, so that $\zeta$ is an algebraic curve in $k^2$.  Let $(x_0,y_0)$ be a smooth point of $\zeta$ where $\partial_{y}P\neq 0$, i.e.~the tangent vector is not vertical.  Without loss of generality (for instance, by applying a translation) we may assume $(x_0,y_0) = (0,0)$. However, we may sometimes refer to the point $(x_0,y_0)$ in order to clarify the role played by this point.

Note first that $P$ can be written as
\beq
P_1(x,y) + P_2(x,y) + \ldots + P_D(x,y)
\eeq
where $P_i$ is a homogeneous polynomial of degree $i$.  We know there is no constant term $P_0$ because $P$ vanishes at $(0,0)$.  The fact that $\zeta$ is smooth at $(0,0)$ is equivalent to the nonvanishing of $P_1 = ax + by$, and the fact that the tangent vector is not vertical implies that $b \neq 0$.  Thus we can think of $P(x,y)$ as a polynomial $Q(y)$ with coefficients in $k[x]$, such that $Q(0) \in k[x]$ vanishes modulo $x$, and $\del_y Q(0) \in k[x]$ reduces to $b \neq 0$ mod $x$.  We have already hypothesized that $Q$ is irreducible in $k[x,y]$.  However, $Q$ admits a factorization in the ring $k[[x]][y]$.  More precisely:  by Hensel's lemma (see e.g. \cite[Thm 7.3]{Eisenbud}), there is a {\em unique} power series $\phi_{(P;x_0,y_0)}(x) \in xk[[x]]$ such that $y = \phi_{(P;x_0,y_0)}(x)$ is a root of $Q$. 

\begin{example}
If $P = x - 2y - y^2 = (1+x) - (1+y)^2$, then
\beq
\phi_{(P;x_0,y_0)}(x) = \sqrt{1+x} - 1 = \frac{1}{2} x - \frac{1}{8} x^2 + \frac{1}{16} x^3 - \frac{5}{128} x^4 + \ldots
\eeq
The other root of $Q$ is $y = 2 - \phi_{(P;x_0,y_0)}(x)$, which is a unit in $k[[x]]$.  We note that the denominators in the power series forbid us from working in characteristic $2$; this is as it should be, since when $\car k = 2$ we have $P = x - y^2$, which has vertical slope at $(0,0)$.
\end{example}

  In other words, we can write
\begin{equation}\label{defnPhi}
P(x,y) = (y - \phi_{(P;x_0,y_0)}(x))R(x,y),
\end{equation}
where $R(x,y)$ is a power series in $k[[x,y]]$ with nonzero constant term.

We now define a sequence of rational functions $f_1, f_2, \ldots \in k(x,y)$ as follows.  We first take
\beq
f_1 = (\del_x P) / (\del_y P).
\eeq

Write $\Delta$ for the linear differential operator defined by 
\beq
\Delta(f) = \del_x f - f_1 \del_y f,
\eeq
and for all $i > 1$, define
\beq
f_i = \Delta f_{i-1}.
\eeq
In geometric terms, we may think of $f_i$ as the derivative of $f_{i-1}$ along the curve $\zeta$.  Each $f_i$ is a rational function in $x$ and $y$ whose degree is bounded in terms of $i$ and $D$.

Define
\beq
\hat{\OO} = k[[x,y]]/P(x,y) = k[[x,y]] / (y - \phi_{(P;x_0,y_0)}(x)),
\eeq
where the second equality follows from the fact that $R(x,y)$ is a unit in $k[[x,y]]$.  $\hat{\OO}$ is called the completed local ring of $\zeta$ at the point $(0,0)$.  If $k=\CC$, we may think of $\hat{\OO}$ as the ring of germs of holomorphic functions at $(0,0)$.  Note, though, that we place no convergence condition on our power series; a series like $\sum n! x^n y^n$ is a perfectly good element of $\CC[[x,y]]$, though there is no complex disc around $0$ in which this power series converges.

By the assumption that $\del_y P$ doesn't vanish at $(0,0)$, we can write $f_1$ as an element of $k[[x,y]]$, and then project it to an element $g_1$ of the quotient $\hat{\OO}$. We note that if $f = P h$, for some function $h \in k[[x,y]]$, then
\beq
\Delta(f) = \del_x (Ph) - f_1 \del_y (Ph)
= (\del_x P) h + P \del_x h - \frac{\del_x P}{\del_y P} ((\del_y P) h + P \del_y h)
= P(\del_x h  - \frac{\del_x P}{\del_y P} \del_y h)
= P\Delta(h).
\eeq
In particular, this means that $\Delta$ preserves the principal ideal $P k[[x,y]]$, and so it descends to a well-defined linear operator (which we continue to denote by $\Delta$) from $\hat{\OO}$ to $\hat{\OO}$.  In particular, if we define $g_i = \Delta g_{i-1}$ for all $i > 1$, then $g_i$ is the projection of $f_i$ to $\hat{\OO}$.  For example, if $k=\CC$ then we may think of $g_i$ as the germ of the rational function $f_i$ in an infinitesimal neighborhood of $(0,0)$ on the curve $\zeta$.

Recalling the factorization $P = (y-\phi_{(P;x_0,y_0)}(x))R$ in $k[[x,y]]$, we have
\beq
\del_y P = (y-\phi_{(P;x_0,y_0)}(x)) \del_y R + R
\eeq
and
\beq
\del_x P = (y-\phi_{(P;x_0,y_0)}(x)) \del_x R - \del_x \phi_{(P;x_0,y_0)}(x) R.
\eeq
Projection to $\hat{\OO}$ sends $y - \phi_{(P;x_0,y_0)}(x)$ to $0$, so
\beq
g_1(x) = - \del_x \phi_{(P;x_0,y_0)}(x).
\eeq
Note that $g_1$ depends on the polynomial $P$, the ``base point'' $(x_0,y_0)$ and the value of $x$ at which it is evaluated. In practice, we will always evaluate $g_1$ at $x=x_0$, but for now it will be helpful to distinguish between $x$ and $x_0$.
Now
\beq
g_2(x) = \Delta g_1(x) = \del_x g_1(x) -  f_1(x,y) \del_y g_1(x) =  - \partial_{x}^2 \phi_{(P;x_0,y_0)}(x)
\eeq
and similarly,
\beq
g_i(x) = -\partial_x^i\phi_{(P;x_0,y_0)}(x)
\eeq
for all positive integers $i$. In particular, setting $x=x_0$, we have
\begin{equation}\label{defnGiX0}
g_i(x_0) = -\partial_x^i  \phi_{(P;x_0,y_0)}(x_0).
\end{equation}

To make the role of $P$ explicit, we will sometimes write $f_{i;P}$ for the rational functions constructed above.  When $P = P_\alpha$, we will write this as $f_{i;P_\alpha}(x,y)$. If we want to refer to numerator and denominator separately, we will call them $F_{i,\alpha}(x,y)$ and $G_{i,\alpha}(x,y)$, so
\begin{equation}\label{defnFG}
f_{i;P_\alpha}(x,y) = \frac{F_{i,\alpha}(x,y)}{G_{i,\alpha}(x,y)},
\end{equation}
where the fraction is understood to be in lowest terms.

\subsection{Space curves modeling  plane curve tangency}
In this section, for each pane curve $\zeta$ we will construct several space curves that capture some of the relevant tangency information of $\zeta$. 

Let $\alpha\in k^{\binom{D+2}{2}}$. Define
\begin{equation}\label{defnBeta1}
\beta_1(\alpha)=\big\{(x,y,z)\colon P_\alpha(x,y)=0,\ z\partial_y P_\alpha(x,y)-\partial_x P_\alpha(x,y)=0\big\}.
\end{equation}

Then $\beta_1(\alpha)$ is an algebraic curve in $k^3$. If $(x,y,z)\in \beta_1(\alpha)$ and if $\partial_y P_\alpha(x,y)\neq 0$, then $z$ is the slope of the curve $\BZ(P_\alpha)$ at the point $(x,y)$. Furthermore, for each $(x,y)\in \BZ(P_\alpha)$ for which $\partial_y P_\alpha(x,y)\neq 0$, there is a unique $z\in k$ so that $(x,y,z)\in\beta_1(\alpha)$. In particular, if $P_\alpha$ is irreducible and if $\partial_y P_\alpha$ does not vanish identically on $\BZ(P_\alpha)$, then the projection $\pi\colon \beta_1(\alpha)\to\BZ(P_\alpha)$ is an isomorphism away from the (finite) set of points on $\BZ(P_\alpha)$ where $\partial_y P_\alpha$ vanishes.  If $p =(x,y)$ is a point of $\BZ(P_\alpha)$ where $\partial_y P_\alpha$ vanishes, there are two situations to consider.  If $\partial_x P_\alpha$ doesn't vanish at $p$, then the fiber of $\beta_1(\alpha)$ over $p$ is empty.  If $\partial_x P_\alpha(p) = 0$, then $(x,y,z)$ lies in $\beta_1(\alpha)$ for all $z$; in other words, the fiber is a vertical line.  We conclude that $\beta_1(\alpha)$ is the union of an irreducible curve $\beta_1^*(\alpha)$ whose projection to the $xy$-plane is an open subset of $\BZ(P_\alpha)$, and a finite set of vertical lines.

Note that $\beta_1^*(\alpha)$ is a component of the intersection of two surfaces of degree $\leq D$, and thus has degree at most $D^2$.

More generally, for each $s$ we wish to define an algebraic space curve $\beta_s^*(\alpha)$ so that for most points $(x,y,z)\in\beta_s(\alpha)$, $z$ corresponds to the ``$(s-1)$--st derivative'' of the slope of $\BZ(P_\alpha)$ at $(x,y)$. Let us now make this precise. Define
\begin{equation}\label{defnBetaN}
\beta_s(\alpha)=\big\{(x,y,z)\colon P_\alpha(x,y)=0,\ z G_{s,\alpha}(x,y)-F_{s,\alpha}(x,y)=0\big\},
\end{equation}
Where $G_{s,\alpha}(x,y)$ and $F_{s,\alpha}(x,y)$ are defined in \eqref{defnFG}. Note that the definition of $\beta_1(\alpha)$ from \eqref{defnBeta1} agrees with the definition from \eqref{defnBetaN}.

As before, if $P_\alpha$ is irreducible and if $\partial_y P_\alpha$ does not vanish identically on $\BZ(P_\alpha)$, we can decompose $\beta_s(\alpha)$ into the union of a finite set of vertical lines plus a unique irreducible component $\beta_s^*(\alpha)$ that is not a vertical line. This irreducible component will be of degree at most $O_{D,s}(1)$. Finally, if $(x,y,z)\in \beta_s(\alpha)$ and if $G_{s,\alpha}(x_0,y_0)\neq 0$, then $z=F_{s,\alpha}(x_0,y_0)/G_{s,\alpha}(x_0,y_0)$. 

\begin{lem}\label{expressingTangentSpace}
Let $\alpha\in k^{\binom{D+2}{2}}$.  Let $(x_0,y_0,z_0)\in\beta_s(\alpha)$, and let $\phi=\phi_{(\alpha;x_0,y_0)}$ be the function from \eqref{defnPhi} associated to $P_\alpha$ at the point $(x_0,y_0)$. Then
\begin{equation}\label{tangVecToBetas}
\Big(1,\ \partial_x \phi(x_0),\ \partial_x^{s+1}\phi(x_0)\Big)\in T_{(x_0,y_0,z_0)}\beta_s(\alpha).
\end{equation}
\end{lem}
\begin{proof}
After a translation, we can assume that $(x_0,y_0)=(0,0)$. Let $\BI(\beta_s(\alpha))$ be the ideal of polynomials in $k[x,y,z]$ that vanish on $\beta_s(\alpha)$. Note that 
\begin{equation*}
\BI(\beta_s(\alpha))=\big( P_\alpha(x,y),\ z G_{s,\alpha}(x,y)-F_{s,\alpha}(x,y)\big).
\end{equation*} 
In particular, the image of $\BI(\beta_s(\alpha))$ in $k[[x,y,z]]$ is $\Big(y-\phi(x),\ z- \partial_x^s\phi(x)\Big).$ Thus, the Zariski cotangent space of $\beta_s(\alpha)$ at $(0,0,0)$ is spanned by the vectors $(-\partial_{x}\phi(0),1,0)$ and $(-\partial_x^{s+1}\phi(0),0,1)$ (these vectors will be linearly independent if and only if $(0,0,0)$ is a smooth point of $\beta_s(\alpha)$, but this fact is not relevant at the moment). Thus $(1,\partial_x\phi(0),\partial_x^{s+1}\phi(0))$ lies in the Zariski tangent space of $\beta_s(\alpha)$ at $(0,0,0).$ Undoing the translation to the origin, we obtain \eqref{tangVecToBetas}. 

\end{proof}

\subsection{Space curves modeling  plane curve orthogonality}
Let $\alpha\in k^{\binom{D+2}{2}}$. Define
\begin{equation}\label{defnTildeBeta1}
\tilde\beta_1(\alpha)=\{(x,y,z)\colon P_\alpha(x,y)=0,\ z\partial_x P_\alpha(x,y)+\partial_y P_\alpha(x,y)=0\}.
\end{equation}
\begin{rem}\label{negRecip}
Compare the definition of $\tilde\beta_1(\alpha)$ with the definition of $\beta_1(\alpha)$ from \eqref{defnBeta1}---if the following conditions are satisfied:
\begin{itemize}
\item $\alpha\in k^{\binom{D+2}{2}}$,
\item $\operatorname{char}(k)=0$ or $D<\operatorname{char}(k)$,
\item $(x,y,z)\in \beta_1^*(\alpha)$,
\item $(x,y,\tilde z)\in \tilde{\beta}_1^*(\alpha)$,
\item  $\partial_yP_\alpha(x,y)\neq 0$,
\item $\partial_xP_\alpha(x,y)\neq 0$,
\end{itemize}
then $\tilde z = -1/z$.
\end{rem}

Again, if $P_\alpha$ is irreducible and if $\partial_x P_\alpha$ does not vanish identically on $\BZ(P_\alpha)$, then the projection $\pi\colon \tilde\beta_1(\alpha)\to\BZ(P_\alpha)$ is an isomorphism away from the (finite) set of points on $\BZ(P_\alpha)$ where $\partial_x P_\alpha$ vanishes. Again,  $\tilde\beta_1(\alpha)$ is the union of an irreducible curve $\tilde\beta_1^*(\alpha)$ whose projection to the $xy$-plane is an open subset of $\BZ(P_\alpha)$, and a finite set of vertical lines. 

The key virtue of $\tilde\beta^*$ is that it translates curve-curve orthogonality in the plane to curve-curve intersection in $k^3$. More precisely, we have the following lemma, which follows from Remark \ref{negRecip}:

\begin{lem}\label{orthIntersectionLem}
Suppose $\operatorname{char}(k)=0$ or $D\leq\operatorname{char}(k)$. Let $P$ and $P^\prime$ be irreducible polynomials of degree $\leq D$ with  $\partial_y P$ (resp. $\partial_x P^\prime$) not vanishing identically on $\BZ(P)$ (resp.  $\BZ(P^\prime)$), and let $(x,y)$ be a point in $\BZ(P)\cap \BZ(P^\prime)$ for which $\partial_y P\neq 0$ and $\partial_x P^\prime\neq 0.$ Then $\BZ(P)$ and $\BZ(P^\prime)$ intersect orthogonally at $(x,y)$ if and only if 
\begin{equation}
(x,y)\in\pi\big(\beta_1^*(\alpha_P)\cap\tilde\beta_1^*(\alpha_{P^\prime})\big).
\end{equation}
\end{lem}

\section{Plane curve tangency and space curve intersection}\label{planeSpaceCurveSec}
In this section we will prove several lemmas that relate tangencies of plane curves with the intersections of the corresponding space curves. These results will be helpful in proving Theorem  \ref{mainThm}.

\begin{lem}\label{directionOfIntersectingCurves}
Let $s\geq 1$. Let $\alpha,\alpha^\prime\in k^{\binom{D+2}{2}}$, with $P_\alpha$ and $P_{\alpha^\prime}$ irreducible. Let $(x_0,y_0,z_0)\in\beta_s(\alpha)\cap \beta_s(\alpha^\prime)$ and suppose 

\begin{equation}\label{blah}
\begin{split}
f_{1,P_\alpha}(x_0,y_0) &= f_{1,P_\alpha^\prime}(x_0,y_0),\\
f_{s+1,P_\alpha}(x_0,y_0) &\neq f_{s+1,P_\alpha^\prime}(x_0,y_0),
\end{split}
\end{equation}
and the denominators of the above rational functions are not $0$. Then 
\begin{equation}
(0,0,1)\in\operatorname{span}\big\{T_{(x_0,y_0,z_0)}\beta_s(\alpha),\ T_{(x_0,y_0,z_0)}\beta_s(\alpha^\prime)\big\}.
\end{equation}
\end{lem}
\begin{proof}
Define 
\begin{equation*}
\begin{split}
&m=f_{1,P_\alpha}(x_0,y_0) = f_{1,P_\alpha^\prime}(x_0,y_0),\\
&m_{s+1} = f_{s+1,P_\alpha}(x_0,y_0) ,\\
&m^{\prime}_{s+1}=f_{s+1,P_\alpha^\prime}(x_0,y_0).
\end{split}
\end{equation*}
By \eqref{blah}, $m_{s+1}\neq m^\prime_{s+1}$. By Lemma \ref{expressingTangentSpace},
\begin{equation*}
\begin{split}
\operatorname{span}\big\{T_{(x_0,y_0,z_0)}\beta_s(\alpha),\ T_{(x_0,y_0,z_0)}\beta_s(\alpha^\prime)\big\}&\supset\operatorname{span}\{(1,m,m_{s+1}),(1,m,m^\prime_{s+1})\big\}\\
&=\operatorname{span}\{(1,m,m_{s+1}),(0,0,m_{s+1}-m^\prime_{s+1})\}\\
&=\operatorname{span}\{(1,m,m_{s+1}),(0,0,1)\}.\qedhere
\end{split}
\end{equation*}
\end{proof}
The following variant of B\'ezout's theorem asserts that two disjoint irreducible curves cannot be too tangent.

\begin{lem}\label{multiplicityLem}  
Let $P,P^\prime\in k[x,y]$ be polynomials of degree $\leq D$. Suppose that either $\operatorname{char}(k)=0$ or $D^2+1<\operatorname{char}(k)$ and that $P$ and $P^\prime$ are irreducible. Let $(x_0,y_0)\in\BZ(P)\cap\BZ(P^\prime)$. Suppose $P$ and $P^\prime$ are smooth at $(x_0,y_0)$ and have non-vertical tangent. Suppose furthermore that
\beq
f_{i,P}(x_0,y_0) = f_{i,P^\prime}(x_0,y_0),\qquad \quad i=1,2,\ldots, D^2 + 1.
\eeq
 Then $P' = cP$ for some constant $c\in k$.
\end{lem}

\begin{proof}  Let $\phi_{P;x_0,y_0}(x), \phi_{P^\prime;x_0,y_0}(x)\in xk[[x]]$ be the power series defined in \eqref{defnPhi}. Write
\beq
\phi_{P,x_0,y_0}(x) = a_1 x + a_2 x^2 + \ldots,
\eeq
and
\beq
\phi_{P^\prime,x_0,y_0}(x) = a'_1 x + a'_2 x^2 + \ldots.
\eeq

Recall that from \eqref{defnGiX0} and the surrounding discussion, we have that in the complete local ring $\hat{\OO} = k[[x,y]]/(P)$, the rational function $f_{i,P}(x,y)$ is equal to $-\partial_x^i \phi_{P,x_0,y_0}(x)$. Thus
\beq
f_{i,P}(0,0) = - i! a_i.
\eeq
Similarly, $f_{i,P'}(0,0) = -i! a_i$.  By the hypothesis on $\car k$, we know $i! \neq 0$ for $i \leq D^2 + 1$, so $a_i = a'_i$ for $i = 1,2, \ldots, D^2+1$.  In other words, $\phi_{P,x_0,y_0}(x) - \phi_{P^\prime,x_0,y_0}(x)$ is divisible by $x^{D^2+1}$.

If $P$ and $P^\prime$ are irreducible and are not multiples of each other, then by B\'ezout's theorem (Theorem \ref{BezoutThm}),
\beq
\dim_k k[[x,y]] / (P,P')\leq D^2.
\eeq

Now $P = (y-\phi_{P,x_0,y_0}(x)) R$ and $P' = (y-\phi_{P^\prime,x_0,y_0}(x))R'$ where $R,R'$ are units in $k[[x,y]]$.  So
\beq
k[[x,y]] / (P,P') = k[[x,y]](y-\phi_{P,x_0,y_0}(x),y-\phi_{P^\prime,x_0,y_0}(x)) = k[[x]]/(\phi_{P,x_0,y_0}(x) - \phi_{P^\prime,x_0,y_0}(x)).
\eeq
We have shown above that  $\phi_{P,x_0,y_0}(x) - \phi_{P^\prime,x_0,y_0}(x)$ lies in $x^{D^2+1} k[[x]]$.  We conclude that
\beq
\dim_k k[[x,y]] / (P,P') \geq D^2+1,
\eeq
which is a contradiction. Thus $P' = cP$, as claimed.
\end{proof}

\begin{rem}\label{F2Remark}
We note that the condition on the characteristic of $k$ is not merely an artifact of the method; not only Lemma~\ref{multiplicityLem} but the main theorem would fail to hold without some such hypothesis.  For instance, consider a family of $n$ curves of the form
\beq
y = a_i x^2 + b_i
\eeq
over $\bar{\mathbb{F}}_2$.  These smooth, irreducible curves have horizontal tangent at every point; in particular, each of the $n^2$ points of intersection between these curves is a point of tangency.
\end{rem}

\section{Proof of Theorem \ref{mainThm}}\label{proofOfMainThmSec}
Let $\PC$ be the set of curves from the statement of Theorem \ref{mainThm}. Suppose that
\begin{equation*}
\sum_{(p,\ell)\in\mathcal{T}(\mathcal{L})}\operatorname{mult}(p,\ell;\mathcal{L})> C_Dn^{3/2}.
\end{equation*}
If $C_D$ is sufficiently large, we will obtain a contradiction. After applying a linear transformation to $k^2$, we can associate an irreducible polynomial $P=P[\zeta]$ to each curve $\zeta\in\PC$, as described in Section \ref{planeCurveDefPolySec}, and since $\operatorname{char}(k)>D$, we can ensure that for each curve $\zeta\in\PC$,  $\partial_y P[\zeta]$ does not vanish identically on $\BZ(P[\zeta])$.  
 
 First, note that for each $\zeta\in\PC$ and each index $j\geq 1$, 
 \begin{equation}\label{badPts}
 |\BZ(P[\zeta])\cap \BZ(G_{j,\alpha_\zeta})|=O_{j,D}(1),
 \end{equation}
 i.e.~there are few points on $\zeta$ where the denominator of $F_{j,\alpha_P}/G_{j,\alpha_P}$ from \eqref{defnFG} vanishes. We say that $p\in\zeta$ is a \emph{good} point of $\zeta$ if $p$ is not in the set \eqref{badPts} for any $j=1,\ldots,D^2$. Otherwise, $p$ is a \emph{bad} point of $\zeta$. Note that 
\begin{equation}\label{fewBadPts}
\sum_{\zeta\in\PC}|\{p\in\zeta\colon\ p\ \textrm{is a bad point of}\ \zeta\}| = O_{D}(n).
\end{equation}

Define $\operatorname{mult}_{g}(p,\ell;\PC)$ to be the number of curves from $\PC$ that are smooth at $p$, tangent to $\ell$ at $p$, and for which $p$ is a good point of the curve. Define $\mathcal{T}^\prime\subset\mathcal{T}(\mathcal{L})$ to be the set of directed points of tangency for which $\operatorname{mult}_{g}(p,\ell;\PC)$ is larger than one. 

By Lemma \ref{multiplicityLem} and pigeonholing, we can find an index $1\leq j\leq D^2$ and a set $\mathcal{T}^{\prime\prime}\subset\mathcal{T}^\prime$ of size $|\mathcal{T}^{\prime\prime}|\geq(2D)^{-2}|\mathcal{T}^\prime|$ so that for each $(p,\ell)\in \mathcal{T}^{\prime\prime}$, there exist curves $\zeta,\ \zeta^\prime$ in $\mathcal{C}$ so that $p$ is a good point for $\zeta$ and $\zeta^\prime$, and \eqref{blah} holds for this value of $j$ with $\alpha=\alpha_\zeta$ and $\alpha^\prime=\alpha_{\zeta^\prime}$. 

Define 
\begin{equation*}
\Gamma=\{ \beta_j^*(\alpha_\zeta)\colon \zeta\in\PC\}.
\end{equation*}
If $\gamma\in\Gamma$, we define $\zeta(\gamma)$ to be the unique element of $\PC$ satisfying $\gamma=\beta_j^*(\alpha_\zeta)$. If $p\in k^2$ is a good point of $\zeta,\zeta^\prime\in \PC$, and $(p,\ell)$ is a directed point of tangency for $\zeta$ and $\zeta^\prime$, then the corresponding space curves $\beta_j^*(\alpha_\zeta)$ and $\beta_j^*(\alpha_{\zeta^\prime})$ intersect at the point $(x,y,z)$, where $p=(x,y)$ and $z$ is the slope of $\ell$. We will abuse notation slightly and identify directed points of tangency with points in $k^3$. 

Consider the graph whose vertices are the curves from $\Gamma$, and two vertices $\gamma,\gamma^\prime$ are adjacent if there is a directed point of tangency $(p,\ell)\in\mathcal{T}^{\prime\prime}$ so that $p$ is a good point for the curves $\zeta(\gamma)$ and $\zeta(\gamma^\prime)$, and \eqref{blah}  holds for $\zeta(\gamma)$ and $\zeta(\gamma^\prime)$ with the value of $j$ specified above. This graph has $n$ vertices and at least $(2D)^{-2}C_Dn^{3/2}$ edges. Thus by \cite[Lemma 2.8]{DG15}, we can find an induced subgraph with $n^\prime\geq n^{1/2}$ vertices such that each vertex has degree at least 
$$
\frac{(2D)^{-2}C_D(n^\prime)^{3/2}}{O(1)n^{\prime}}=\frac{C_D}{O_D(1)}(n^\prime)^{1/2}.
$$ 
Let $\Gamma^\prime\subset\Gamma$ be the set of curves corresponding to the vertices of the induced subgraph. 

Define $\PC^\prime=\{\zeta(\gamma)\colon \gamma\in\Gamma^\prime\}$. Let  $R\in k[x,y,z]$ be a non-zero polynomial of minimal degree that vanishes on all the curves in $\Gamma^\prime$. 

\begin{lem}
\begin{equation}\label{degRSmall}
\deg R =O_D\big((n^\prime)^{1/2}\big).
\end{equation}
\end{lem} 
\begin{proof}
Recall that each curve in $\Gamma^\prime$ has degree $O_D(1)$. Let $C_D^\prime$ be a constant to be chosen later, and let $\mathcal{P}$ be a set of $C_D^\prime(n^\prime)^{3/2}$ points, with $C_D^\prime(n^\prime)^{1/2}$ points on each curve of $\mathcal{L}^\prime$. Let $S$ be a polynomial of degree $\leq 100(C_D^\prime)^{1/3} (n^\prime)^{1/2}$ that vanishes on the points of $\mathcal{P}.$ By B\'ezout's theorem, if $C_D^\prime$ is chosen sufficiently large (depending only on $D$), then $S$ vanishes on every curve in $\mathcal{L}^\prime.$ Since $R$ is the polynomial of lowest degree that vanishes on the curves in $\mathcal{L}^\prime$, we have $\deg R\leq \deg S$, which gives us \eqref{degRSmall}. 
\end{proof}
In particular, if the constant $c_D$ from the statement of Theorem \ref{mainThm} is chosen sufficiently small, then 
\begin{equation}\label{prodDegreesSmall}
(\deg\gamma)(\deg R)<\operatorname{char}(k)
\end{equation}
for every curve $\gamma\in\Gamma^\prime$.

\begin{lem}\label{containedInBzR}
 If $c_D$ is sufficiently small and $C_D$ is sufficiently large (depending only of $D$), then every curve in $\Gamma^\prime$ is contained in $\BZ(\partial_z R)$.
\end{lem}
\begin{proof}
Fix a curve $\gamma\in\Gamma^\prime$. There are at least $\frac{C_D}{O_D(1)} (n^\prime)^{1/2}$ distinct points $(x,y,z)\in\gamma$ so that the following holds
\begin{itemize}
\item[(i)] $(x,y,z)$ is a good point of $\gamma$.
\item[(ii)] There is a second curve $\gamma^\prime\in\Gamma$ so that \eqref{blah} holds at $(x,y)$ (with the value of $j$ determined above)  for $\alpha$ and $\alpha^\prime$ corresponding to $\gamma$ and $\gamma^\prime$, respectively.
\item[(iii)] $(x,y,z)$ is a good point of $\gamma^\prime$.
\end{itemize}
Let $(x,y,z)$ be one such point and let $\gamma^{\prime}\in\Gamma^\prime$ be the curve from Item (ii) above. By Lemma \ref{directionOfIntersectingCurves}, the span of the tangent spaces of $\beta_k^*(\alpha)$ and $\beta_k^*(\alpha^\prime)$ at $(x,y,z)$ contains the vector $(0,0,1)$. This implies that $\partial_zR(x,y,z)=0$. This means that there are at least $\frac{C_D}{O_D(1)} (n^\prime)^{1/2}$ points on $\gamma$ for which $\partial_zR=0$. If $C_D$ is selected sufficiently large (depending only on $D$), then the number of points on $\gamma$ at which $\partial_z R$ vanishes is greater than $(\deg\gamma)(\deg \partial_z R)$. Since $(\deg\gamma)(\deg R)<\operatorname{char}(k)$, we can apply B\'ezout's theorem to conclude that $\gamma\subset \BZ(\partial_z R)$. 
\end{proof}
Since $R$ was a minimal degree non-zero polynomial whose zero-set contained all the curves from $\Gamma^\prime$, Lemma \ref{containedInBzR} implies that $\partial_z R = 0$.  In particular, if $x^a y^b z^c$ is a monomial appearing with nonzero coefficient in $R$, we have $c x^a y^b z^{c-1} = 0$, so $c$ is congruent to $0$ mod $\operatorname{char}(k)$; but $c \leq \deg R < \operatorname{char}(k)$, so this implies that $c=0$.  In other words,  $R(x,y,z)=S(x,y)$ for some polynomial $S\in k[x,y]$. However, each curve in $\PC^\prime$ must be a distinct irreducible component of $\BZ(S)$, and this implies 
\begin{equation}\label{degSBig}
\deg(R)=\deg(S) \geq |\PC^\prime|=n^{\prime}. 
\end{equation}
If we select the constant $C_D$ in Theorem \ref{mainThm} sufficiently large, then this contradicts \eqref{degRSmall}. This completes the proof of Theorem \ref{mainThm}.\qedhere

\section{Orthogonal curves}\label{orthCurvesSec}
Before proving Theorem \ref{orthThm}, we will introduce some additional tools and terminology for dealing with space curves in $k^3$. An important tool will be Proposition \ref{guthZahlProp}, which describes the structure of arrangements of space curves that determine many two-rich points. In this section (as in previous sections), we will always assume that $k$ is an algebraically closed field.

\subsection{Constructible sets}
A constructible set is a generalization of an algebraic variety. Algebraic varieties are defined by finite sets of polynomial equalities; for example the set $\{w\in k^d\colon P_1(w)=0,\ldots, P_s(w)=0\}$ is an algebraic variety. (Affine) constructible sets are defined by finite sets of polynomial equalities and non-equalities; for example the set 
\begin{equation*}
\{w\in k^d\colon P_1(w)=0,\ldots, P_s(w)=0,\ Q_1(w)\neq 0,\ldots, Q_t(w)\neq 0\} 
\end{equation*}
is constructible. If $X$ and $Y$ are constructible subsets of $k^d$, then so are $X\cup Y,\ X\cap Y$, and $X\backslash Y$. Similarly, (projective) constructible sets are defined by finite sets of homogeneous polynomial equalities and non-equalities. In particular, the set of irreducible degree $\leq D$ plane curves in $k^2$ is a constructible subset of $\mathbf{P}k^{\binom{D+2}{2}}.$

The degree of an algebraic variety is one measure of how ``complicated'' that variety is. The analogous notation for constructible sets is called complexity. See \cite[Chapter 3]{Harris} for an introduction to constructible sets, and see \cite{GuZa} for further information about the complexity of constructible sets. 
\subsection{The Chow variety of curves}
We will work with a parameter space of algebraic space curves called the Chow variety. Normally, this is a projective variety that parameterizes projective curves in $k^3$, but for our purposes it will be easier to work with a slightly smaller object which is an affine variety that parameterizes (affine) curves in $k^3$.

For each $D\geq 1$, let $\mathcal{C}_D$ be the (affine) Chow variety of irreducible (affine) curves of degree $\leq D$ in $k^3$. $\mathcal{C}_D$ is a constructible set of complexity $O_D(1)$, and there is a constructible set $\tilde{\mathcal{C}}_D\subset\mathcal{C}_D\times k^3$ of complexity $O_D(1)$ that is equipped with two projections $\pi_1\colon\tilde{\mathcal{C}}_D\to\mathcal{C}_D$ and $\pi_2\colon\tilde{\mathcal{C}}_D\to k^3$ with the following properties:
\begin{itemize}
 \item For each $w\in\mathcal{C}_D$, $\pi_2\circ \pi^{-1}_1(w)=\gamma_w$ is an irreducible algebraic curve in $k^3$.
 \item For ``every'' irreducible algebraic curve $\gamma\subset k^3$ of degree $\leq D$, there is a unique point $w_\gamma\in\mathcal{C}_D$ so that $\gamma=\pi_2\circ \pi^{-1}_1(w_\gamma)$.
\end{itemize}
 We put the word ``every'' in quotes because when constructing the affine Chow variety, we intersect the true Chow variety (which is a projective variety) with the compliment of a generic hyperplane. Thus the above statement fails for a small number of curves $\gamma\subset k^3$. However, since the original problem we are interested in concerns a finite set of curves, this subtlety will not affect us. For more information about the Chow variety, see \cite[Chapter 7]{Harris} or \cite{GeMo}; both of these sources describe the projective version of the Chow variety, and they only deal with the Chow variety of irreducible curves of degree exactly $D$. For a precise construction of the affine Chow variety $\mathcal{C}_D$ used here, see \cite{GuZa}. Henceforth, we will abuse notation slightly and use the notation $\gamma\in\mathcal{C}_D$ to refer to curves in the Chow variety.

\subsection{Sets of curves with many rich points}
Informally, Theorem 3.7 from \cite{GuZa} says that if an arrangement of space curves (taken from a particular family of curves) determines many two-rich points, then many of these curves must be contained in a low degree doubly-ruled surface. In order to make this precise, we will need a definition.

\begin{defn}[Doubly ruled surface]\label{defnDoublyRuled}
Let $\mathcal{C}\subset\mathcal{C}_{D}$ be a constructible set. Let $Z\subset k^3$ be an irreducible surface. We say that $Z$ is doubly ruled by curves from $\mathcal{C}$ if there is a Zariski open set $O\subset Z$ so that for every $w\in O$, there are at least two distinct curves from $\mathcal{C}$ containing $w$ and contained in $Z$.
\end{defn}

\begin{thm}[\cite{GuZa} Theorem 3.7, special case]\label{guthZahlProp}
Let $D\geq 1$ and let $\mathcal{C}\subset\mathcal{C}_D$ be a constructible set of curves. Then there exist constants $c_{D,\mathcal{C}}>0$ (small) and $C_{D,\mathcal{C}}$ (large) so that the following holds. Let $\Gamma\subset\mathcal{C}$ be a finite set of curves of cardinality $n$, with $n\leq c_{D,\mathcal{C}}\operatorname{char}(k)^2$ (if $\operatorname{char}(k)=0$ then we place no restriction on $n$). Then at least one of the following must hold:
\begin{itemize}
 \item The number of two-rich points is $\leq C_{D,\mathcal{C}} n^{3/2}$.
 \item There is an irreducible surface $Z\subset k^3$ of degree $\leq 100D^2$ that is doubly ruled by curves from $\mathcal{C}$; this surface contains $\geq c_{D,\mathcal{C}}n^{1/2}$ curves from $\Gamma$. Furthermore, for every $m\geq 1$, we can find two disjoint sets of curves, each of cardinality $m$, so that every curve from the first set intersects every curve from the second set.
 \end{itemize}
\end{thm}
We will not require the full strength of Theorem \ref{guthZahlProp}. In particular, we will never use the fact that the doubly ruled surface $Z$ (if it exists) has small degree, nor will we use the fact that it contains $\Omega(n^{1/2})$ curves from $\Gamma$.

\subsection{Space curves respecting tangency conditions}

\begin{lem}
Let $W\subset k^{\binom{D+2}{2}}$ be a constructible set of degree $\leq D$ curves. Then for each $D^\prime$, the sets
 \begin{equation*}
 \begin{split}
&\mathcal{C}_{X,D^\prime}= \{ \gamma\in\mathcal{C}_{D^\prime}\colon \gamma =  \beta_1^*(\alpha)\ \textrm{for some}\ \alpha\in W \},\\
&\tilde{\mathcal{C}}_{X,D^\prime}= \{ \gamma\in\mathcal{C}_{D^\prime}\colon \gamma =  \tilde{\beta}_1^*(\alpha)\ \textrm{for some}\ \alpha\in W \}
\end{split}
\end{equation*}
are constructible and have complexity $O_{W,D^\prime}(1)$. 
 \end{lem}
\begin{proof}
We will begin with $\mathcal{C}_{X,D^\prime}$. Define
 \begin{equation*}
\begin{split}
Y_{D^\prime}=&\big\{(\gamma,(x,y,z),\alpha)\in\mathcal{C}_{D^\prime} \times k^3\times W\colon (x,y,z)\in\gamma,\ P_\alpha(x,y)=0,\\
&\qquad\qquad\qquad\qquad\qquad\partial_y P_\alpha(x,y)=0\ \textrm{OR}\ z \partial_y P_\alpha(x,y)- \partial_x P_\alpha(x,y)=0\big\},\\
Y_{D^\prime}^\prime=& \{(\gamma,\alpha)\in\mathcal{C}_{D^\prime} \times W\colon \gamma\ \textrm{is not a vertical line},\ (\gamma,(x,y,z),\alpha)\in Y_{D^\prime}\ \textrm{for all}\ (x,y,z)\in \gamma\}.
\end{split}
\end{equation*}
$Y_{D^\prime}$ is constructible of complexity $O_{D^\prime, W}(1)$; the condition $(x,y,z)\in\gamma$ can be written as $(\gamma,(x,y,z))\in \tilde{\mathcal{C}}_{D^\prime}$, and the latter is a constructible set of complexity $O_{D^\prime}(1)$. To verify that $Y_{D^\prime}^\prime$ is constructible, observe that the set 
\begin{equation}\label{zxAlphaSet}
A= \{(\gamma,(x,y,z),\alpha)\in\mathcal{C}_D \times k^3\times W\colon\ (x,y,z)\in\gamma,\ (\gamma,(x,y,z),\alpha)\notin Y_{D^\prime}\}
\end{equation}
is constructible and has complexity $O_{D^\prime,W}(1)$. Let $\pi_1(\gamma,(x,y,z),\alpha)=(\gamma,\alpha)$ be the projection map. Then
\begin{equation*}
 \pi_1(A)= \{(\gamma,\alpha)\in\mathcal{C}_{D^\prime} \times W\colon\ \textrm{there exists}\ (x,y,z)\in\gamma,\ \textrm{such that}\ (\gamma,(x,y,z),\alpha)\notin Y_{D^\prime}\}.
\end{equation*}
By Chevalley's theorem (see \cite[Theorem 3.16]{Harris}), $ \pi_1(A)$ is also constructible and has complexity $O_{D^\prime,W}(1)$. Finally, 
\begin{equation*}
Y_{D^\prime}^\prime=\big(\{\gamma\in\mathcal{C}_{D^\prime}:\gamma\ \textrm{is not a vertical line}\} \times W\big)\ \backslash\  \pi_1(A). 
\end{equation*}
Let $\pi_2\colon \mathcal{C}_{D^\prime}\times W\to\mathcal{C}_{D^\prime}$ be the projection map. Then $\mathcal{C}_{D^\prime,W}=\pi_2(Y_{D^\prime}^\prime)$; this establishes that $\mathcal{C}_{D^\prime,W}$ is constructible and has complexity $O_{D^\prime,W}(1)$. An identical argument shows that $\tilde{\mathcal{C}}_{W,D^\prime}$ is constructible and has complexity $O_{W,D^\prime}(1)$.
\end{proof}

Define the set of monic irreducible polynomials,
\begin{equation*}
\operatorname{MI}_{D}=\{\alpha\in k^{\binom{D+2}{2}}\colon P_{\alpha}\ \textrm{is irreducible and $x$--monic}\}.
\end{equation*}

Note that $\operatorname{MI}_{D}$ is a constructible set of complexity $O_{D}(1)$, and there is an injection $\varphi\colon \operatorname{MI}_{D}\to \mathbf{P}k^{\binom{D+2}{2}}$ that sends the point $(w_1,\ldots,w_{\binom{D+2}{2}})$ to the point $[1:w_1:\ldots:w_{\binom{D+2}{2}}]$. We will be interested in $\varphi^{-1}$. 

Let $X\subset \mathbf{P}k^{\binom{D+2}{2}}$ be a family of curves and define 
\begin{equation}\label{defnXPrime}
\operatorname{Aff}(X) = \varphi^{-1}(\varphi(\operatorname{MI}_{D})\cap X). 
\end{equation}
(This is the ``affinization'' of the projective set $X$). Then $\operatorname{Aff}(X)\subset\operatorname{MI}_{D}$ is constructible and has complexity $O_{D,X}(1)$. Note that for each $\alpha\in \operatorname{Aff}(X)$, there is a unique element $\gamma\in\mathcal{C}_{\operatorname{Aff}(X),D^2}$  that satisfies $\gamma=\beta_1^*(\alpha)$. This is because $\beta_1^*(\alpha)$ is an irreducible curve of degree at most $D^2$. Similarly, there is a unique element $\gamma\in\tilde{\mathcal{C}}_{\operatorname{Aff}(X),D^2}$  that satisfies $\gamma=\tilde\beta_1^*(\alpha)$. We now have the necessary tools to prove Theorem \ref{orthThm}.
\section{Proof of Theorem \ref{orthThm}}
\begin{proof}
Let $\mathcal{L}$ be a set of $n$ irreducible curves in $X$. By applying a suitable invertible linear transformation $T\colon k^2\to k^2$ if necessary, we can guarantee that $w\in \varphi(\operatorname{MI}_{D})$ for each $w\in \mathcal{L}$. The family $X$ is replaced by the family $\{ T(\zeta)\colon\zeta\in X \};$ this is also a family of curves of degree $\leq D$. Abusing notation slightly, we will call this new family $X$ as well.

Let $X^\prime=\operatorname{Aff}(X)$, and let $\mathcal{C} = \mathcal{C}_{X^\prime,D^2}\cup  \tilde{\mathcal{C}}_{X^\prime,D^2}$. This is a family of degree $\leq D^2$ curves of complexity $O_{D,X}(1)$. Define
\begin{equation*}
\begin{split}
&\Gamma = \{\beta_1^*(\alpha)\colon \alpha \in \varphi^{-1}(\mathcal{L})\},\\
&\tilde\Gamma = \{\tilde{\beta}_1^*(\alpha)\colon \alpha \in \varphi^{-1}(\mathcal{L})\}.
\end{split}
\end{equation*}
$\Gamma$ and $\tilde\Gamma$ are sets of space curves that encode the slope of the corresponding plane curves from $\mathcal{L}$. $\Gamma$ and $\tilde\Gamma$ are finite subsets of $\mathcal{C}$. Let $\Gamma_0=\Gamma\cup\tilde\Gamma.$

Apply Theorem \ref{guthZahlProp} to $\Gamma_0$. Either there are $O_{X}(n^{3/2})$ two-rich points determined by the arrangement of curves in $\Gamma_0$, or for any $m\geq 1$, we can find a set $\Gamma^*\subset\mathcal{C}$ of cardinality $m$ that determines $\geq m^2/4$ two-rich points. 

If the first option occurs, then by Lemma \ref{orthIntersectionLem}, the curves from $\PC$ determine $O_{X}(n^{3/2})$ directed points of orthogonality, and we are done.

Suppose the second option occurs. Write $\Gamma^*=(\Gamma^*\cap \mathcal{C}_{X^\prime,D^2}) \cup(\Gamma^*\cap \tilde{\mathcal{C}}_{X^\prime,D^2})$. Each two-rich point determined by $\Gamma^*\cap \mathcal{C}_{X^\prime,D^2}$ corresponds to a directed point of tangency amongst the curves in
\begin{equation*}
\{\BZ(P_\alpha)\colon \beta_1(\alpha)\in  \Gamma^*\cap \mathcal{C}_{X^\prime,D^2}\}. 
\end{equation*}
By Theorem \ref{mainThm}, there are $O_D(m^{3/2})$ such points. Similarly, each two-rich point determined by $\Gamma^*\cap \tilde{\mathcal{C}}_{X^\prime,D^2}$ corresponds to a directed point of tangency amongst the curves in \begin{equation*}
\{\BZ(P_\alpha)\colon \tilde{\beta}_1(\alpha)\in \Gamma^*\cap \tilde{\mathcal{C}}_{X^\prime,D^2}\}. 
\end{equation*}
Again, by Theorem \ref{mainThm}, there are $O_D(m^{3/2})$ such points.

Define
\begin{equation*}
\PC_m=\{\BZ(P_\alpha)\colon \beta_1(\alpha)\in \Gamma^*\cap \mathcal{C}_{X^\prime,D^2}\} \cup \{\BZ(P_\alpha)\colon \tilde{\beta}_1(\alpha)\in \Gamma^*\cap \tilde{\mathcal{C}}_{X^\prime,D^2}\}.
\end{equation*}

There are $\Omega_{X}(m^2)$ points that are hit by at least one curve from $\Gamma^*\cap  \mathcal{C}_{X^\prime,D^2}$ and at least one curve from $\Gamma^*\cap \tilde{\mathcal{C}}_{X^\prime,D^2}$. By Lemma  \ref{orthIntersectionLem}, each of these points correspond to a distinct directed point of orthogonality from the curves in the arrangement $\PC_m$. Therefore we have constructed an arrangement of $m$ curves from the family $X$ that determine $(m/2)^2(1-o_{X}(1))$ directed points of orthogonality.
\end{proof}

\section{Generalizations and further directions}

\subsection{Improved bounds}\label{improvedBoundsSection}
It is natural to ask whether Theorem \ref{mainThm} is sharp. Remark \ref{F2Remark} shows that if $k$ has characteristic two, then it is impossible to obtain nontrivial bounds on the number of directed tangencies. If $n\sim (\operatorname{char}(k))^2$, then Theorem \ref{mainThm} is best possible. For example, let $\mathcal{L}$ be the set of all unit circles in $\mathbb{F}_p^2$; then $|\mathcal{L}|=|\mathbb{F}_p^2|\sim (\operatorname{char}(\mathbb{F}_p))^2$, and for each point $p\in \mathbb{F}_p^2$ and each line $\ell$ passing through $p$, there are exactly two unit circles that are smooth at $p$ and tagent to $\ell$ at $p$. Thus no three curves from $\mathcal{L}$ are tangent at a common point, and $\mathcal{L}$ determines $\sim p^3=|\mathcal{L}|^{3/2}$ tangencies, so the bound from Theorem \ref{mainThm} is attained (technically, Theorem \ref{mainThm} only applies if $|\mathcal{L}|$ is less than $c_D(\operatorname{char}(\mathbb{F}_p))^2$, (here $D=2$) but we can meet this requirement by randomly selecting a subset of $\mathcal{L}$, where each circle is chosen with probability $\sim c_D$).

 However, if $|\mathcal{L}|$ is much smaller than $(\operatorname{char}(k))^2$, then we do not know if the bound from Theorem \ref{mainThm} is sharp. For example, if $k=\RR$ or $k=\CC$, the best construction we are aware of attains a lower bound of $\Omega(n^{4/3})$. In brief, the construction is as follows. Let $\mathcal{P}$ be a set of $n$ points in $\RR^2$ and let $\mathcal{L}_0$ be a set of $n$ lines in $\RR^2$ that determine $\sim n^{4/3}$ point-line incidences. Let $\mathcal{L}_1$ be the set of $n$ unit circles centered at the points of $\mathcal{P},$ and let $\mathcal{L}_2$ be the set of $n$ lines obtained by translating each line $\ell\in\mathcal{L}_1$ one unit in the direction $\ell^\perp$. Then if $(p,\ell)$ is a point-line incidence from the original collection of points and lines, the corresponding circle and line will be tangent. Let $\mathcal{L}=\mathcal{L}_1\cup\mathcal{L}_2$; this is a set of $2n$ irreducible curves of degree $\leq 2$ that determines $\sim n^{4/3}$ tangencies. By applying an inversion transform centered at a point that does not lie on any of the circles or lines, one can transform $\mathcal{L}$ into a set of $2n$ circles that determine $\sim n^{4/3}$ tangencies. See Chapter 7.1 of \cite{BMP} for further details. The same example can be realized in $\CC^2$ by taking the complexification of the curves from $\mathcal{L}.$ 
\subsection{Higher-order tangency}
We believe that an analogue of Theorem \ref{mainThm} should hold for higher-order tangencies. As the order of tangency increases, the number of such tangencies should decrease. Heuristically, the number of ``$j$--th order tangencies'' spanned by an arrangement of $n$ curves should be $O(n^{(j+2)/(j+1)})$. We will make this notion more precise below. 

\begin{defn}[Jet Bundle]
Let $k$ be a field and let $1\leq s<\operatorname{char}(k)$ (if $\operatorname{char}(k)=0$, we only require $s\geq 1$). We define the bundle of one-dimensional $s$--jets in $k^2$ to be the set of equivalence classes 
\begin{equation*}
\mathcal{J}_s=\{\overline{(p,\tau)}\colon p\in\ k^2,\ \tau\ \textrm{is a curve through $p$ that is smooth at}\ p\},
\end{equation*}
where $\overline{(p,\tau)}=\overline{(p^\prime,\tau^\prime)}$ if $p=p^\prime$ and $\tau$ and $\tau^\prime$ intersect at $p$ with multiplicity $\geq s$. If $s=1$, then $\mathcal{J}_1$ can be identified with the set of tuples $(p,\ell)$, where $p\in k^2$ and $\ell$ is a line passing through $p$.  
\end{defn}

\begin{defn}[Directed points of higher order tangency]
Let $k$ be a field and let $1\leq s<\operatorname{char}(k)$ (if $\operatorname{char}(k)=0$, we only require $s\geq 1$). Let $\PC$ be a set of irreducible algebraic curves in $k^2$. Let $\overline{(p,\tau)}\in \mathcal{J}_s$. We say that $\overline{(p,\tau)}$ is a directed point of $s$--th order tangency for $\PC$ if there are at least two distinct curves in $\PC$ that are smooth at $p$ and that intersect $\tau$ at $p$ with multiplicity $\geq s$ (crucially, this does not depend on the choice of representative $\tau$ from the equivalence class $\overline{(p,\tau)}$). 

Define $\mathcal{T}_s(\mathcal{L})$ be the set of directed points of $s$--th order tangency, and for each $\overline{(p,\tau)}\in\mathcal{T}(\mathcal{L}),$ define $\operatorname{mult}_s(\overline{(p,\tau)};\mathcal{L})$ to be the number of curves from $\mathcal{L}$ that are smooth at $p$ and intersect $\tau$ at $p$ with multiplicity $\geq s$. 
\end{defn}

The natural analogue of the space curves described in Section \ref{planeSpaceCurvesSec} would be curves in $\RR^{d+s}$.  This leads us to conjecture the following bound on the number of higher-order tangencies.

\begin{conj}
Fix $D,s\geq 1$. Let $\mathcal{L}$ be a set of $n$ irreducible algebraic curves of degree $\leq D$, with $n\leq c_{D,s}\operatorname{char}(k)^{s+1}$. Then
\begin{equation}
\sum_{(p,\gamma)\in\mathcal{T}_j(\mathcal{L})}\operatorname{mult}_s(\overline{(p,\tau)};\mathcal{L})\leq C_{D,s} n^{(s+2)/(s+1)}.
\end{equation}
\end{conj}

\subsection{Higher-dimensional hypersurfaces}
We conjecture than an analogue of Theorem \ref{mainThm} should be true for hypersurfaces in higher dimension. 
\begin{defn}[Higher dimensional directed points of tangency]
Let $k$ be a field and let $\mathcal{S}$ be a set of irreducible hypersurfaces in $k^d$. Let $p$ be a point in $k^d$ and let $H$ be a hyperplane containing $p$. We say that $(p,H)$ is a directed point of tangency for $\mathcal{S}$ if there are at least two distinct hypersurfaces in $\mathcal{S}$ that are smooth at $\zeta$ and tangent to $H$ at $\zeta$. Define $\mathcal{T}(\mathcal{S})$ be the set of directed points of tangency, and for each $(p,H)\in\mathcal{T}(\mathcal{S}),$ define $\operatorname{mult}(p,H;\mathcal{S})$ to be the number of hypersurfaces from $\mathcal{S}$ that are smooth at $p$ and tangent to $H$ at $p$.
\end{defn}

The natural analogue of the space curves described in Section \ref{planeSpaceCurvesSec} would be $d-1$ dimensional varieties in $\RR^{2d-1}$. This leads us to conjecture the following higher-dimensional analogue of  Theorem \ref{mainThm}.

\begin{conj}\label{higherDim}
Let $\mathcal{S}$ be a set of $n$ irreducible hypersurfaces in $k^d$ of degree at most $D$, with $n\leq c_D\operatorname{char}(k)^{d}$. Then
\begin{equation}
\sum_{(p,H)\in\mathcal{T}(\mathcal{S})}\operatorname{mult}(p,H;\mathcal{S})\leq C_Dn^{(2d-1)/d}.
\end{equation}
\end{conj}




\bibliographystyle{amsplain}
\providecommand{\bysame}{\leavevmode\hbox to3em{\hrulefill}\thinspace}
\providecommand{\MR}{\relax\ifhmode\unskip\space\fi MR }
\providecommand{\MRhref}[2]{%
  \href{http://www.ams.org/mathscinet-getitem?mr=#1}{#2}
}
\providecommand{\href}[2]{#2}


\begin{dajauthors}
\begin{authorinfo}[JSE]
  Jordan S. Ellenberg\\
  University of Wisconsin, Madison\\
  Madison, WI\\
  ellenber\imageat{}math\imagedot{}wisc\imagedot{}edu
\end{authorinfo}
\begin{authorinfo}[JS]
  Jozsef Solymosi\\
  University of British Colombia\\
  Vancouver, BC\\
  solymosi\imageat{}math\imagedot{}ubc\imagedot{}ca
\end{authorinfo}
\begin{authorinfo}[JZ]
  Joshua Zahl\\
  University of British Colombia\\
  Vancouver, BC\\
  jzahl\imageat{}math\imagedot{}ubc\imagedot{}ca
\end{authorinfo}
\end{dajauthors}

\end{document}